\documentclass[reqno,11pt,centertags]{amsart}
\usepackage{geometry}
\usepackage{amsmath,amsthm,amscd,amssymb,latexsym,upref,stmaryrd}
\usepackage{graphicx}
\usepackage{hyperref}
\usepackage{enumitem}
\usepackage{xcolor}
\newcommand*{\mailto}[1]{\href{mailto:#1}{\nolinkurl{#1}}}
\usepackage[numbers,sort&compress]{natbib}

\newcommand{\beq}{\begin{equation}}
	\newcommand{\eeq}{\end{equation}}
\newcommand{\ba}{\begin{align}}
	\newcommand{\ea}{\end{align}}

 \allowdisplaybreaks
\numberwithin{equation}{section}

\newtheorem{theorem}{Theorem}[section]
\newtheorem{lemma}[theorem]{Lemma}

\theoremstyle{definition}

\begin{document}
	
	
	\title[Incomplete Inverse problem for Dirac operators]
	{Incomplete inverse problem for Dirac operator with constant delay}
	
	\author[F.~Wang]{Feng Wang}
	\address{School of Mathematics and Statistics, Nanjing University of
		Science and Technology, Nanjing, 210094, Jiangsu, China}
	\email{\mailto{wangfengmath@njust.edu.cn}}
	
	\author[C.~F.~Yang]{CHUAN-FU Yang}
	\address{Department of Mathematics, School of Mathematics and Statistics, Nanjing University of
		Science and Technology, Nanjing, 210094, Jiangsu, People's
		Republic of China}
	\email{\mailto{chuanfuyang@njust.edu.cn}}

	\date{\today}
	
	\begin{abstract}
{In this work, we consider Dirac-type operators with a constant delay less than two-fifths of the interval and not less than one-third of the interval. For our considered Dirac-type operators, an incomplete inverse spectral problem is studied. Specifically, when two complex potentials are known a priori on a certain subinterval, reconstruction of the two potentials on the entire interval is studied from complete spectra of two boundary value problems with one common boundary condition. The uniqueness of the solution of the inverse problem is proved. A constructive method is developed for the solution of the inverse problem.}
	\end{abstract}
	
	\maketitle
	
\section{introduction and main results}
	In the past decade, there appeared a significant interest in inverse problems for Sturm-Liouville-type operators with constant delay:
\begin{equation}\label{0.1}
  -y''(x)+q(x)y(x-a)=\lambda y(x),\quad x\in(0,\pi),
\end{equation}
under two-point boundary conditions (see [2-3, 7-12, 14, 19, 23-25, 29-30] and references therein),  which are often adequate for modelling various real-world processes frequently possessing a nonlocal nature. Here $q(x)$ is a complex-valued function in $L_{2}(a,\pi)$ vanishing on $(0,a)$. It is well known that the potential $q(x)$ is uniquely determined by specifying the spectra of two boundary value problems for equation (\ref{0.1}) with a common boundary condition at zero as soon as $a\in[\frac{2\pi}{5},\pi)$. The recent series of papers [9-11] establishes, however, that it is never possible for $a\in(0,\frac{2\pi}{5})$. For more details, see Introduction in \cite{B-M-S}.

To the best of our knowledge, the first attempt of defining Dirac-type operator with constant delay was made in \cite{Buter2}. They consider the following Dirac-type system with a delay constant $a\!\in\!(0,\pi)$:

\begin{align}\label{1.1}
  By'(x)+Q(x)y(x-a)=\lambda y(x),\quad x\in (0,\pi),
\end{align}
where
\begin{align}
  B=\begin{bmatrix}
    0 & 1   \\  -1 & 0
  \end{bmatrix},\;\;
  Q(x)=\begin{bmatrix}
    q(x) & p(x)   \\  p(x) & -q(x)
  \end{bmatrix},\;\;
   y(x)=\begin{bmatrix}
    y_{1}(x) \\ y_{2}(x)
  \end{bmatrix}, \nonumber
\end{align}
while $q(x)$ and $p(x)$ are complex-valued functions belong to $L_{2}(0,\pi)$, and $Q(x)=0$ on $(0,a)$. For any fixed $j\in\{1,2\}$, denote by $L_{j}(Q)$ the boundary value problem for equation (\ref{1.1}) with the boundary conditions
\begin{align}\label{1.2}
  y_{1}(0)=0,\quad   y_{j}(\pi)=0.
\end{align}

In \cite{Buter2}, authors restrict themselves to the case $a\in[\frac{\pi}{2}, \pi)$, when the dependence of the characteristic functions of the problems $L_{j}(Q)$ on $Q(x)$ is linear. For the considered case, however, they give answers to the full range of questions usually raised in the inverse spectral theory. Specifically, reconstruction of two complex potentials $q$ and $p$ is studied from either complete spectra or subspectra of two boundary value problems $L_{1}(Q)$ and $L_{2}(Q)$. They give conditions on the subspectra that are necessary and sufficient for the unique determination of the potentials. Moreover, necessary and sufficient conditions for the solvability of both inverse problems are obtained. For the inverse problem of recovering from the complete spectra, they establish also uniform stability in each ball of a finite radius.

In \cite{W-Y3}, we restrict ourselves to the case $a\in[\frac{2\pi}{5},\frac{\pi}{2})$, when the dependence of the characteristic functions of the problems $L_{j}(Q)$ on $Q(x)$ is nonlinear. For our considered case, we study the inverse problem of restoring two complex potentials $q$ and $p$ from complete spectra of two boundary value problems $L_{1}(Q)$ and $L_{2}(Q)$. We also provide full answers for our considered inverse problem: uniqueness, solvability and uniform stability.

In \cite{D-V0}, by constructing counterexamples, authors prove that two spectra of the problems $L_{1}(Q)$ and $L_{2}(Q)$ cannot uniquely determine the potentials $q$ and $p$ in the case $a\in[\frac{\pi}{3},\frac{2\pi}{5})$. Therefore, in this case, a natural question is whether it is possible to add some information of the potential functions so that two spectra of the problems $L_{1}(Q)$ and $L_{2}(Q)$ can also uniquely determine the potentials $q$ and $p$. This  article aims to address this issue.

In this paper, we restrict ourselves to the case $a\in[\frac{\pi}{3},\frac{2\pi}{5})$. For this case, an incomplete inverse spectral problem is studied. Specifically, when the two complex potentials $q$ and $p$ are known a priori on a certain subinterval, reconstruction of the two potentials on the entire interval is studied from two spectra of the problems $L_{1}(Q)$ and $L_{2}(Q)$ (for details, see Inverse Problem  1 below). A uniqueness theorem is given for the incomplete inverse problem (See Theorem \ref{th2} below). As well as, a constructive method is developed for the solution of the incomplete inverse problem (for details, see Algorithm 1 in Section 3). The motivation for studying this incomplete inverse problem comes from the paper \cite{D-V}, in which a similar inverse problem is studied for the Sturm-Liouville-type operators with constant delay.

Additionally, in the classical case $a=0$, inverse problems for (\ref{1.1}) were studied in [1, 15-18, 20-22] and other works. Meanwhile, it is worth mentioning that the first attempt of defining operator (\ref{0.1}) on a star-type graph was made in \cite{W-Y1, W-Y2}, which could be classified as locally nonlocal because the delay on each edge does not affect the other edges. Afterwards, Buterin suggests another concept of operator  with delay on graphs that can be characterized as globally nonlocal, when the delay extends through vertices of the graph (for details, please refer to \cite{Buter}).

To begin with, let us give the following theorem about the asymptotic relations of eigenvalues for the boundary value problems $L_{j}(Q)$, $j=1,2$, which has been proved in \cite{W-Y3}.

\begin{theorem}\label{th1}
For $j\in\{1,2\}$, the boundary value problem $L_{j}(Q)$ has infinitely many eigenvalues $\lambda_{n,j}$, $n\in\mathbb{Z}$, of the form
\begin{align}\label{1.3}
\lambda_{n,j}=n-\frac{j-1}{2}+\kappa_{n,j},\quad \{\kappa_{n,j}\}_{n\in\mathbb{Z}}\in l_{2}.
\end{align}
\end{theorem}

Throughout the article, we use the symbol $f|_{\mathcal{S}}$ for denoting the restriction of the function $f$ to the set $\mathcal{S}$. Assuming that the delay constant $a\in[\frac{\pi}{3},\frac{2\pi}{5})$ is known a priori, we consider the following inverse problem.

\textbf{Inverse Problem 1.}  Given $q|_{(\frac{3a}{2},\frac{\pi}{2}+\frac{a}{4})}$, $p|_{(\frac{3a}{2},\frac{\pi}{2}+\frac{a}{4})}$
and the two spectra $\{\lambda_{n,j}\}_{n\in\mathbb{Z}}$, $j=1,2$, find the potential functions $q(x)$ and $p(x)$ for $x\in[a,\pi]$.

Next, we present the uniqueness theorem for Inverse Problem 1, which is also the main result of this paper. To this end, for $j=1,2$, along with the problem $L_{j}(Q)$, we will consider other problem $L_{j}(\widetilde{Q})$ of the same form but with a different potential matrix
\begin{equation}
  \widetilde{Q}(x)=\begin{bmatrix}
    \widetilde{q}(x) & \widetilde{p}(x)   \\  \widetilde{p}(x) & -\widetilde{q}(x)
  \end{bmatrix}. \nonumber
\end{equation}
We agree that if a certain symbol $\alpha$ denotes an object related to the problem $L_{j}(Q)$, then  this symbol with tilde $\widetilde{\alpha}$ will denote the analogous object related to the problem $L_{j}(\widetilde{Q})$ .

\begin{theorem}\label{th2}
 If $q(x)\!=\!\widetilde{q}(x)$, $p(x)\!=\!\widetilde{p}(x)$ a.e. on $(\frac{3a}{2},\frac{\pi}{2}\!+\!\frac{a}{4})$, and
 $\lambda_{n,1}\!=\!\widetilde{\lambda}_{n,1}$, $\lambda_{n,2}\!=\!\widetilde{\lambda}_{n,2}$, $n\in\mathbb{Z}$, then $q(x)\!=\!\widetilde{q}(x)$ and $p(x)\!=\!\widetilde{p}(x)$ a.e. on $[a,\pi]$. In other words, if the potentials $q(x)$ and $p(x)$  are known a priori on the subinterval $(\frac{3a}{2},\frac{\pi}{2}\!+\!\frac{a}{4})$, then the specification of two spectra $\{\lambda_{n,1}\}_{n\in\mathbb{Z}}$ and $\{\lambda_{n,2}\}_{n\in\mathbb{Z}}$ uniquely determines the potentials $q$ and $p$ on the interval $[a,\pi]$.
\end{theorem}

The paper is organized as follows. In the next section,  we introduce the characteristic functions of the boundary value problems $L_{1}(Q)$ and $L_{2}(Q)$. In Section 3, we prove Theorem \ref{th2} and provide a constructive method for the solution of Inverse Problem 1.

\section{characteristic functions}
	
	Let $Y(x,\lambda)$ be the fundamental matrix-solution of equation (\ref{1.1}) such that
\begin{align}
 Y(x,\lambda)=\begin{bmatrix}
    y_{1,1}(x,\lambda) & y_{1,2}(x,\lambda)   \\  y_{2,1}(x,\lambda) & y_{2,2}(x,\lambda)
  \end{bmatrix},\;\;
 Y(0,\lambda)=\begin{bmatrix}
    1 & 0   \\  0 & 1
  \end{bmatrix}. \nonumber
\end{align}
Then, for $j\in\{1,2\}$, eigenvalues of the problem $L_{j}(Q)$ coincide with zeros of the entire function
\begin{align}\label{2.1}
  \Delta_{j}(\lambda)=y_{j,2}(\pi,\lambda),
\end{align}
which is called characteristic function of $L_{j}(Q)$.

When $a\in[\frac{\pi}{3},\frac{\pi}{2}$), it follows from the relation (12) in \cite{Buter2} that
\begin{align}\label{2.3}
  Y(x,\lambda)=Y_{0}(x,\lambda)+Y_{1}(x,\lambda)+Y_{2}(x,\lambda),
\end{align}
where
\begin{align}\label{2.4}
  Y_{0}(x,\lambda)=\begin{bmatrix}
   \cos\lambda x & -\sin\lambda x   \\  \sin\lambda x  & \cos\lambda x
  \end{bmatrix},
\end{align}
\begin{align}\label{2.5}
  Y_{1}(x,\lambda)=B\int_{a}^{x}Y_{0}(x-t,\lambda)Q(t)Y_{0}(t-a,\lambda)dt,
\end{align}
\begin{align}\label{2.6}
  Y_{2}(x,\lambda)=B\int_{2a}^{x}Y_{0}(x-t,\lambda)Q(t)Y_{1}(t-a,\lambda)dt.
\end{align}

Combining these relations (\ref{2.3})-(\ref{2.6}) and taking the definition (\ref{2.1}) into account, by a rather tedious computation, we obtain
\begin{align}
 \Delta_{1}(\lambda)\!=\!-\sin\!\lambda\pi\!-\!\!\int_{a}^{\pi}\!\!q(t)\cos\!\lambda(\pi\!-\!2t\!+\!a)dt
 \!+\!\!\int_{a}^{\pi}\!\!p(t)\sin\!\lambda(\pi\!-\!2t\!+\!a)dt \nonumber\\
 -\int_{2a}^{\pi}dt\int_{a}^{t-a}\Big(q(t)q(s)+p(t)p(s)\Big)\sin\!\lambda(\pi\!-\!2t\!+\!2s)ds \qquad\;\,\nonumber\\
 +\int_{2a}^{\pi}dt\int_{a}^{t-a}\Big(q(t)p(s)-p(t)q(s)\Big)\cos\!\lambda(\pi\!-\!2t\!+\!2s)ds, \qquad\,\nonumber
\end{align}
and
\begin{align}
 \Delta_{2}(\lambda)\!=\!\cos\!\lambda\pi\!-\!\!\int_{a}^{\pi}\!\!q(t)\sin\!\lambda(\pi\!-\!2t\!+\!a)dt
 \!-\!\!\int_{a}^{\pi}\!\!p(t)\cos\!\lambda(\pi\!-\!2t\!+\!a)dt \nonumber \;\;\\
 +\int_{2a}^{\pi}dt\int_{a}^{t-a}\Big(q(t)p(s)-p(t)q(s)\Big)\sin\!\lambda(\pi\!-\!2t\!+\!2s)ds \qquad\;\nonumber\\
 +\int_{2a}^{\pi}dt\int_{a}^{t-a}\Big(q(t)q(s)+p(t)p(s)\Big)\cos\!\lambda(\pi\!-\!2t\!+\!2s)ds. \qquad\!\nonumber
\end{align}
Changing the variable and interchanging the order of integration, we obtain
\begin{align}\label{2.7}
 \Delta_{1}(\lambda)\!=\!-\sin\!\lambda\pi\!+\!\!\int_{a-\pi}^{\pi-a}\!\!v_{1}(x)\sin\!\lambda xdx\!
 +\!\!\int_{a-\pi}^{\pi-a}\!\!v_{2}(x)\cos\!\lambda xdx,
\end{align}
and
\begin{align}\label{2.8}
 \Delta_{2}(\lambda)\!=\!\cos\!\lambda\pi\!+\!\!\int_{a-\pi}^{\pi-a}\!\!v_{2}(x)\sin\!\lambda xdx
 \!-\!\int_{a-\pi}^{\pi-a}\!\!v_{1}(x)\cos\!\lambda xdx,\;
\end{align}
where
\begin{align}\label{2.9}
v_{1}(x)=\begin{cases}
\frac{1}{2}p(\frac{\pi+a-x}{2})\!-\!\frac{1}{2}\int_{\frac{\pi+2a-x}{2}}^{\pi}\Big[q(t)q(\frac{x+2t-\pi}{2})  \\
\qquad\qquad\qquad+p(t)p(\frac{x+2t-\pi}{2})\Big]dt, \;\: x\!\in\!(2a\!-\!\pi,\pi\!-\!2a),\\
\\
\frac{1}{2}p(\frac{\pi+a-x}{2}),\;\:x\!\in[a\!-\!\pi,2a\!-\!\pi]\cup[\pi\!-\!2a,\pi\!-\!a],
\end{cases}
\end{align}
\begin{align}\label{2.10}
v_{2}(x)=\begin{cases}
-\frac{1}{2}q(\frac{\pi+a-x}{2})\!+\!\frac{1}{2}\int_{\frac{\pi+2a-x}{2}}^{\pi}\Big[q(t)p(\frac{x+2t-\pi}{2})  \\
\qquad\qquad\qquad-p(t)q(\frac{x+2t-\pi}{2})\Big]dt, \;\: x\!\in\!(2a\!-\!\pi,\pi\!-\!2a),\\
\\
-\frac{1}{2}q(\frac{\pi+a-x}{2}),\;\:x\!\in[a\!-\!\pi,2a\!-\!\pi]\cup[\pi\!-\!2a,\pi\!-\!a].
\end{cases}\!
\end{align}
Using Euler's formula, the relations (\ref{2.7}) and (\ref{2.8}) take the forms
\begin{align}\label{2.11}
 \Delta_{1}(\lambda)=-\sin\lambda\pi+\int_{a-\pi}^{\pi-a}u_{1}(x)\exp(i\lambda x)dx
\end{align}
and
\begin{align}\label{2.12}
 \Delta_{2}(\lambda)=\cos\lambda\pi+\int_{a-\pi}^{\pi-a}u_{2}(x)\exp(i\lambda x)dx,
\end{align}
where
\begin{align}\label{2.13}
 u_{1}(x)=\frac{v_{1}(x)-v_{1}(-x)}{2i}+\frac{v_{2}(x)+v_{2}(-x)}{2},
\end{align}
\begin{align}\label{2.14}
 u_{2}(x)=\frac{v_{2}(x)-v_{2}(-x)}{2i}-\frac{v_{1}(x)+v_{1}(-x)}{2}.
\end{align}

The following lemma has been obtained (see Lemma 3.3 in \cite{W-Y3}).

\begin{lemma}\label{lem1}
The characteristic functions $\Delta_{1}(\lambda)$ and $\Delta_{2}(\lambda)$ are uniquely determined by specifying their zeros. Moreover, the following representations hold:
\begin{align}\label{2.15}
  \Delta_{1}(\lambda)=\pi(\lambda_{0,1}-\lambda)\prod_{|n|\in\mathbb{N}}\frac{\lambda_{n,1}-\lambda}{n}\exp\left(\frac{\lambda}{n}\right),\;\;
\end{align}
\begin{align}\label{2.16}
  \Delta_{2}(\lambda)=\prod_{n\in\mathbb{Z}}\frac{\lambda_{n,2}-\lambda}{n-\frac{1}{2}}\exp\left(\frac{\lambda}{n-\frac{1}{2}}\right).\qquad\qquad\!
\end{align}
\end{lemma}

From Lemma \ref{lem1}, we immediately obtain the following lemma, which plays an important role in proving Theorem \ref{th2}.

\begin{lemma}\label{lem2}
If $\lambda_{n,1}=\widetilde{\lambda}_{n,1}$, $\lambda_{n,2}=\widetilde{\lambda}_{n,2}$, $n\in\mathbb{Z}$, then $u_{1}(x)=\widetilde{u}_{1}(x)$ and $u_{2}(x)=\widetilde{u}_{2}(x)$ a.e. on $[a-\pi,\pi-a]$.
\end{lemma}

\begin{proof}
Since $\lambda_{n,1}=\widetilde{\lambda}_{n,1}$, $\lambda_{n,2}=\widetilde{\lambda}_{n,2}$, $n\in\mathbb{Z}$, according to Lemma \ref{lem1}, we have
\begin{equation}
  \Delta_{1}(\lambda)=\widetilde{\Delta}_{1}(\lambda),\quad \Delta_{2}(\lambda)=\widetilde{\Delta}_{2}(\lambda),\nonumber
\end{equation}
which yields
\begin{equation}
  \Delta_{1}(\lambda)\!+\!\sin\lambda\pi\!=\!\widetilde{\Delta}_{1}(\lambda)\!+\!\sin\lambda\pi,\quad \Delta_{2}(\lambda)\!-\!\cos\lambda\pi\!=\!\widetilde{\Delta}_{2}(\lambda)\!-\!\cos\lambda\pi.\nonumber
\end{equation}
From (\ref{2.11}) and (\ref{2.12}), one has
\begin{align}
  \int_{a-\pi}^{\pi-a}u_{1}(x)\exp(i\lambda x)dx=\int_{a-\pi}^{\pi-a}\widetilde{u}_{1}(x)\exp(i\lambda x)dx,\nonumber\\
  \int_{a-\pi}^{\pi-a}u_{2}(x)\exp(i\lambda x)dx=\int_{a-\pi}^{\pi-a}\widetilde{u}_{2}(x)\exp(i\lambda x)dx.\nonumber
\end{align}
Taking $\lambda=n\in\mathbb{Z}$ in the above two equations, we have
\begin{align}\label{2.17}
  \int_{a-\pi}^{\pi-a}u_{1}(x)\exp(in x)dx=\int_{a-\pi}^{\pi-a}\widetilde{u}_{1}(x)\exp(in x)dx,
\end{align}
\begin{align}\label{2.18}
  \int_{a-\pi}^{\pi-a}u_{2}(x)\exp(in x)dx=\int_{a-\pi}^{\pi-a}\widetilde{u}_{2}(x)\exp(in x)dx.
\end{align}
Since the system $\{\exp(inx)\}_{n\in\mathbb{Z}}$ is complete in $L_{2}[a-\pi,\pi-a]$, the relations (\ref{2.17}) and (\ref{2.18}) imply
$u_{1}(x)\!=\!\widetilde{u}_{1}(x),\; u_{2}(x)\!=\!\widetilde{u}_{2}(x)$, a.e. on $[a-\pi,\pi-a]$.
\end{proof}

\section{Proof of Theorem 1.2}
	
	Before proceeding directly to the proof of Theorem \ref{th2}, we fulfil some preparatory work. Let
\begin{align}\label{3.1}
  w_{1}(x)=-(u_{1}+iu_{2})(\pi+a-2x)-(u_{1}-iu_{2})(2x-\pi-a),\!\!\!\!
\end{align}
\begin{align}\label{3.2}
  w_{2}(x)=(iu_{1}-u_{2})(\pi+a-2x)-(iu_{1}+u_{2})(2x-\pi-a),
\end{align}
then the functions $u_{1}(x)$ and $u_{2}(x)$ in $L_{2}(a-\pi,\pi-a)$ uniquely determine the functions $w_{1}(x)$ and $w_{2}(x)$ in $L_{2}(a,\pi)$.

When $x\!\in\![a\!-\!\pi,2a\!-\!\pi]\cup[\pi\!-\!2a,\pi\!-\!a]$, according to (\ref{2.9})-(\ref{2.10}) and (\ref{2.13})-(\ref{2.14}), one can calculate \begin{align}
  -2(u_{1}(x)+iu_{2}(x))=(q\!+\!ip)\!\left(\frac{\pi\!+\!a\!-\!x}{2}\right),\nonumber
\end{align}
\begin{align}
  -2(u_{1}(x)-iu_{2}(x))=(q\!-\!ip)\!\left(\frac{\pi\!+\!a\!+\!x}{2}\right).\nonumber
\end{align}
The changes of variables $t=\frac{\pi+a-x}{2}$ and $t=\frac{\pi+a+x}{2}$, respectively, lead to
\begin{align}
(q\!+\!ip)(t)=-2(u_{1}+iu_{2})\!(\pi\!+\!a\!-\!2t),\quad t\in[a,\frac{3a}{2}]\cup[\pi-\frac{a}{2},\pi],\nonumber
\end{align}
\begin{align}
(q\!-\!ip)(t)=-2(u_{1}-iu_{2})\!(2t\!-\!\pi\!-\!a),\quad t\in[a,\frac{3a}{2}]\cup[\pi-\frac{a}{2},\pi].\nonumber
\end{align}
Summing up two equations above and subtracting one from the other, and taking (\ref{3.1})-(\ref{3.2}) into account, we get
\begin{align}
  q(t)=w_{1}(t),\quad p(t)=w_{2}(t),\quad t\in[a,\frac{3a}{2}]\cup[\pi-\frac{a}{2},\pi].\nonumber
\end{align}

Thus, we have proved the following lemma.

\begin{lemma}\label{lem3}
The following relations hold:
\begin{equation}\label{3.3}
  q|_{[a,\frac{3a}{2}]\cup[\pi-\frac{a}{2},\pi]}=w_{1}|_{[a,\frac{3a}{2}]\cup[\pi-\frac{a}{2},\pi]},\;
  p|_{[a,\frac{3a}{2}]\cup[\pi-\frac{a}{2},\pi]}=w_{2}|_{[a,\frac{3a}{2}]\cup[\pi-\frac{a}{2},\pi]}.
\end{equation}
\end{lemma}

\qquad

When $x\!\in\!(2a\!-\!\pi,\pi\!-\!2a)$, according to (\ref{2.9})-(\ref{2.10}) and (\ref{2.13})-(\ref{2.14}), one can calculate
\begin{align}
  -2(u_{1}(x)+iu_{2}(x))=(q\!+\!ip)\!\left(\frac{\pi\!+\!a\!-\!x}{2}\right) \qquad\qquad\qquad\qquad\qquad\nonumber\\
  -\!\!\int_{\frac{\pi+2a-x}{2}}^{\pi}\!(q(t)\!+\!ip(t))(iq\!+\!p)\!\!\left(\frac{x\!+\!2t\!-\!\pi}{2}\right)dt,\!\!\nonumber
\end{align}
\begin{align}
  -2(u_{1}(x)-iu_{2}(x))=(q\!-\!ip)\!\left(\frac{\pi\!+\!a\!+\!x}{2}\right)\qquad\qquad\qquad\qquad\qquad\nonumber\\
  -\!\!\int_{\frac{\pi+2a+x}{2}}^{\pi}\!(q(t)\!-\!ip(t))(-iq\!+\!p)\!\!\left(\frac{2t\!-\!x\!-\!\pi}{2}\right)dt.\!\!\!\!\!\!\nonumber
\end{align}
The changes of variables $\xi=\frac{\pi+a-x}{2}$ and $\xi=\frac{\pi+a+x}{2}$, respectively, lead to
\begin{align}
  -2(u_{1}+iu_{2})(\pi+a-2\xi)\!=\!(q\!+\!ip)(\xi) \qquad\qquad\qquad\qquad\qquad\;\;\;\;\nonumber\\
  -\!\int_{\xi+\frac{a}{2}}^{\pi}\!(q(t)\!+\!ip(t))(iq\!+\!p)\!\!\left(t\!-\!\xi\!+\!\frac{a}{2}\right)\!dt,\nonumber
\end{align}
\begin{align}
  -2(u_{1}-iu_{2})(2\xi-\pi-a)\!=\!(q\!-\!ip)(\xi) \qquad\qquad\qquad\qquad\qquad\;\;\nonumber\\
  -\!\int_{\xi+\frac{a}{2}}^{\pi}\!(q(t)\!-\!ip(t))(-iq\!+\!p)\!\!\left(t\!-\!\xi\!+\!\frac{a}{2}\right)\!dt\!\!\!\!\!\!\nonumber
\end{align}
for $\xi\!\in\!(\frac{3a}{2},\pi\!-\!\frac{a}{2})$. Summing up two equations above and subtracting one from the other, and taking (\ref{3.1})-(\ref{3.2}) into account, we get
\begin{align}\label{3.4}
  w_{1}(\xi)=q(\xi)\!-\!\int_{\xi+\frac{a}{2}}^{\pi}\left[q(t)p\!\left(t\!-\!\xi\!+\!\frac{a}{2}\right)\!-\!p(t)q\!\left(t\!-\!\xi\!
  +\!\frac{a}{2}\right)\right]dt,
\end{align}
\begin{align}\label{3.5}
  w_{2}(\xi)=p(\xi)\!-\!\int_{\xi+\frac{a}{2}}^{\pi}\left[q(t)q\!\left(t\!-\!\xi\!+\!\frac{a}{2}\right)\!+\!p(t)p\!\left(t\!-\!\xi\!
  +\!\frac{a}{2}\right)\right]dt
\end{align}
for $\xi\!\in\!(\frac{3a}{2},\pi\!-\!\frac{a}{2})$.

When $\xi\!\in\![\pi\!-\!a,2a]\cup[\frac{\pi}{2}\!+\!\frac{3a}{4},\pi\!-\!\frac{a}{2})$, it is easy to find that
\begin{align}
  \pi-\frac{a}{2}\leq\xi+\frac{a}{2}\leq t\leq\pi,\quad  a\leq t-\xi+\frac{a}{2}\leq\frac{3a}{2}.   \nonumber
\end{align}
Hence, according to Lemma \ref{lem3}, the relations (\ref{3.4}) and (\ref{3.5}) yield
\begin{align}
  q(\xi)=w_{1}(\xi)\!+\!\int_{\xi+\frac{a}{2}}^{\pi}\left[w_{1}(t)w_{2}\!\left(t\!-\!\xi\!+\!\frac{a}{2}\right)\!-\!w_{2}(t)w_{1}\!\left(t\!-\!\xi\!
  +\!\frac{a}{2}\right)\right]dt, \nonumber
\end{align}
\begin{align}
  p(\xi)=w_{2}(\xi)\!+\!\int_{\xi+\frac{a}{2}}^{\pi}\left[w_{1}(t)w_{1}\!\left(t\!-\!\xi\!+\!\frac{a}{2}\right)\!+\!w_{2}(t)w_{2}\!\left(t\!-\!\xi\!
  +\!\frac{a}{2}\right)\right]dt. \nonumber
\end{align}

Thus, we have proved the following lemma.

\begin{lemma}\label{lem4}
The following relations hold:
\begin{align}\label{3.6}
\begin{cases}
  q|_{[\pi-a,2a]\cup[\frac{\pi}{2}+\frac{3a}{4},\pi-\frac{a}{2})}
  =(w_{1}+\gamma_{1})|_{[\pi-a,2a]\cup[\frac{\pi}{2}+\frac{3a}{4},\pi-\frac{a}{2})},\\
  p|_{[\pi-a,2a]\cup[\frac{\pi}{2}\!+\!\frac{3a}{4},\pi\!-\!\frac{a}{2})}
  =(w_{2}+\gamma_{2})|_{[\pi-a,2a]\cup[\frac{\pi}{2}+\frac{3a}{4},\pi-\frac{a}{2})},
\end{cases}
\end{align}
where
\begin{align}\label{3.7}
\begin{cases}
  \gamma_{1}(x)=\int_{x+\frac{a}{2}}^{\pi}\left[w_{1}(t)w_{2}\!\left(t\!-\!x\!+\!\frac{a}{2}\right)\!-\!w_{2}(t)w_{1}\!\left(t\!-\!x\!
  +\!\frac{a}{2}\right)\right]dt,\!\!\\
  \gamma_{2}(x)=\int_{x+\frac{a}{2}}^{\pi}\left[w_{1}(t)w_{1}\!\left(t\!-\!x\!+\!\frac{a}{2}\right)\!+\!w_{2}(t)w_{2}\!\left(t\!-\!x\!
  +\!\frac{a}{2}\right)\right]dt
\end{cases}\!\!\!\!\!\!\!\!
\end{align}
for $x\in[\pi\!-\!a,2a]\cup[\frac{\pi}{2}\!+\!\frac{3a}{4},\pi\!-\!\frac{a}{2})$.
\end{lemma}

Let $I_{1}=[a,\frac{3a}{2}]\cup[\pi-a,2a]\cup[\frac{\pi}{2}\!+\!\frac{3a}{4},\pi]$. Since $a\in[\frac{\pi}{3},\frac{2\pi}{5})$, three subintervals $[a,\frac{3a}{2}]$, $[\pi-a,2a]$ and $[\frac{\pi}{2}\!+\!\frac{3a}{4},\pi]$ are not intersecting with each other. By combining Lemmas \ref{lem2}, \ref{lem3} and \ref{lem4}, and taking (\ref{3.1})-(\ref{3.2}) into account, we can obtain the following lemma.

\begin{lemma}\label{lem5}
If $\lambda_{n,1}=\widetilde{\lambda}_{n,1}$, $\lambda_{n,2}=\widetilde{\lambda}_{n,2}$, $n\in\mathbb{Z}$, then $q|_{I_{1}}=\widetilde{q}|_{I_{1}}$ and $p|_{I_{1}}=\widetilde{p}|_{I_{1}}$. Thus, the specification of two spectra $\{\lambda_{n,1}\}_{n\in\mathbb{Z}}$ and $\{\lambda_{n,2}\}_{n\in\mathbb{Z}}$ uniquely determines the potentials $q$ and $p$ on the interval $I_{1}$.
\end{lemma}

\begin{proof}
Since $\lambda_{n,1}=\widetilde{\lambda}_{n,1}$, $\lambda_{n,2}=\widetilde{\lambda}_{n,2}$, $n\in\mathbb{Z}$, according to Lemma \ref{lem2}, we have
\begin{align}\label{3.8}
  u_{1}(x)=\widetilde{u}_{1}(x),\quad u_{2}(x)=\widetilde{u}_{2}(x),\quad a.e. \;on \;[a-\pi,\pi-a].\!\!\!
\end{align}
In view of (\ref{3.1}), (\ref{3.2}) and (\ref{3.8}), we get
\begin{align}\label{3.9}
  w_{1}(x)=\widetilde{w}_{1}(x),\quad w_{2}(x)=\widetilde{w}_{2}(x),\quad a.e. \;on \;[a,\pi].\qquad\;
\end{align}
According to Lemma \ref{lem3}, it follows from (\ref{3.9}) that
\begin{align}
  q(x)=\widetilde{q}(x),\quad p(x)=\widetilde{p}(x),\quad a.e. \;on \;[a,\frac{3a}{2}]\cup[\pi-\frac{a}{2},\pi].\!\!\!\!\!\!\!\!\nonumber
\end{align}
According to Lemma \ref{lem4}, it follows from (\ref{3.9}) that
\begin{align}
  q(x)=\widetilde{q}(x),\quad p(x)=\widetilde{p}(x),\quad a.e. \;on \;[\pi-a,2a].\qquad\;\nonumber
\end{align}
So, we have $q|_{I_{1}}=\widetilde{q}|_{I_{1}}$ and $p|_{I_{1}}=\widetilde{p}|_{I_{1}}$.
\end{proof}

Next, for simplicity, we denote
\begin{align}\label{3.10}
f_{1}=q|_{[\frac{\pi}{2}+\frac{3a}{4},\pi]},\qquad f_{2}=p|_{[\frac{\pi}{2}+\frac{3a}{4},\pi]},
\end{align}
When $\xi\!\in\!I_{2}:=(2a,\frac{\pi}{2}\!+\!\frac{3a}{4})$, it is easy to find that
\begin{align}
  \frac{\pi}{2}+\frac{3a}{4}\leq\xi+\frac{a}{2}\leq t\leq\pi,\quad  a\leq t-\xi+\frac{a}{2}\leq\frac{3a}{2}.   \nonumber
\end{align}
Hence, according to Lemma \ref{lem3} and the equation (\ref{3.10}), the relations (\ref{3.4}) and (\ref{3.5}) yield

\begin{align}
  q(\xi)=w_{1}(\xi)\!+\!\int_{\xi+\frac{a}{2}}^{\pi}\left[f_{1}(t)w_{2}\!\left(t\!-\!\xi\!+\!\frac{a}{2}\right)\!-\!f_{2}(t)w_{1}\!\left(t\!-\!\xi\!
  +\!\frac{a}{2}\right)\right]dt, \nonumber
\end{align}
\begin{align}
  p(\xi)=w_{2}(\xi)\!+\!\int_{\xi+\frac{a}{2}}^{\pi}\left[f_{1}(t)w_{1}\!\left(t\!-\!\xi\!+\!\frac{a}{2}\right)\!+\!f_{2}(t)w_{2}\!\left(t\!-\!\xi\!
  +\!\frac{a}{2}\right)\right]dt. \nonumber
\end{align}

Thus, we have proved the following lemma.

\begin{lemma}\label{lem6}
The following relations hold:
\begin{align}\label{3.11}
  q|_{I_{2}}=(w_{1}+\delta_{1})|_{I_{2}},\quad
  p|_{I_{2}}=(w_{2}+\delta_{2})|_{I_{2}},
\end{align}
where
\begin{align}\label{3.12}
\begin{cases}
  \delta_{1}(x)=\int_{x+\frac{a}{2}}^{\pi}\left[f_{1}(t)w_{2}\!\left(t\!-\!x\!+\!\frac{a}{2}\right)\!-\!f_{2}(t)w_{1}\!\left(t\!-\!x\!
  +\!\frac{a}{2}\right)\right]dt,\!\!\\
  \delta_{2}(x)=\int_{x+\frac{a}{2}}^{\pi}\left[f_{1}(t)w_{1}\!\left(t\!-\!x\!+\!\frac{a}{2}\right)\!+\!f_{2}(t)w_{2}\!\left(t\!-\!x\!
  +\!\frac{a}{2}\right)\right]dt,
\end{cases}\!\!\!\!\!
\end{align}
for $x\in I_{2}$.
\end{lemma}

Using Lemmas \ref{lem2}, \ref{lem5} and \ref{lem6}, we immediately obtain the following lemma.

\begin{lemma}\label{lem7}
If $\lambda_{n,1}=\widetilde{\lambda}_{n,1}$, $\lambda_{n,2}=\widetilde{\lambda}_{n,2}$, $n\in\mathbb{Z}$, then $q|_{I_{2}}=\widetilde{q}|_{I_{2}}$ and $p|_{I_{2}}=\widetilde{p}|_{I_{2}}$. Thus, the specification of two spectra $\{\lambda_{n,1}\}_{n\in\mathbb{Z}}$ and $\{\lambda_{n,2}\}_{n\in\mathbb{Z}}$ uniquely determines the potentials $q$ and $p$ on the interval $I_{2}$.
\end{lemma}

\begin{proof}
Since $\lambda_{n,1}=\widetilde{\lambda}_{n,1}$, $\lambda_{n,2}=\widetilde{\lambda}_{n,2}$, $n\in\mathbb{Z}$, according to Lemmas \ref{lem2} and \ref{lem5}, and taking the equations (\ref{3.1}), (\ref{3.2}) and (\ref{3.10}) into account, we have
\begin{align}\label{3.13}
  w_{1}(x)=\widetilde{w}_{1}(x),\quad w_{2}(x)=\widetilde{w}_{2}(x),\quad a.e. \;on \;[a,\pi],
\end{align}
and
\begin{align}\label{3.14}
  f_{1}(x)=\widetilde{f}_{1}(x),\quad f_{2}(x)=\widetilde{f}_{2}(x),\quad a.e. \;on \;[\frac{\pi}{2}\!+\!\frac{3a}{4},\pi].\!\!\!\!\!\!
\end{align}
Using Lemma \ref{lem6}, it is from (\ref{3.13}) and (\ref{3.14}) that
\begin{align}
  q(x)=\widetilde{q}(x),\quad p(x)=\widetilde{p}(x),\quad a.e. \;on \;I_{2}, \nonumber
\end{align}
i.e. $q|_{I_{2}}=\widetilde{q}|_{I_{2}}$ and $p|_{I_{2}}=\widetilde{p}|_{I_{2}}$.
\end{proof}

For simplicity, we denote
\begin{align}\label{3.15}
g_{1}=q|_{[a,\frac{\pi}{2}+\frac{a}{4})},\qquad g_{2}=p|_{[a,\frac{\pi}{2}+\frac{a}{4})},
\end{align}
When $\xi\!\in\!I_{3}:=[\frac{\pi}{2}\!+\!\frac{a}{4},\pi\!-\!a)$, it is easy to find that
\begin{align}
  \frac{\pi}{2}+\frac{3a}{4}\leq\xi+\frac{a}{2}\leq t\leq\pi,\quad  a\leq t-\xi+\frac{a}{2}\leq\frac{\pi}{2}+\frac{a}{4}.   \nonumber
\end{align}
Hence, according to the equations (\ref{3.10}) and (\ref{3.15}), the relations (\ref{3.4}) and (\ref{3.5}) yield

\begin{align}
  q(\xi)=w_{1}(\xi)\!+\!\int_{\xi+\frac{a}{2}}^{\pi}\left[f_{1}(t)g_{2}\!\left(t\!-\!\xi\!+\!\frac{a}{2}\right)\!-\!f_{2}(t)g_{1}\!\left(t\!-\!\xi\!
  +\!\frac{a}{2}\right)\right]dt, \nonumber
\end{align}
\begin{align}
  p(\xi)=w_{2}(\xi)\!+\!\int_{\xi+\frac{a}{2}}^{\pi}\left[f_{1}(t)g_{1}\!\left(t\!-\!\xi\!+\!\frac{a}{2}\right)\!+\!f_{2}(t)g_{2}\!\left(t\!-\!\xi\!
  +\!\frac{a}{2}\right)\right]dt. \nonumber
\end{align}

Thus, we have proved the following lemma.

\begin{lemma}\label{lem8}
The following relations hold:
\begin{align}\label{3.16}
  q|_{I_{3}}=(w_{1}+\eta_{1})|_{I_{3}},\quad
  p|_{I_{3}}=(w_{2}+\eta_{2})|_{I_{3}},
\end{align}
where
\begin{align}\label{3.17}
\begin{cases}
  \eta_{1}(x)=\int_{x+\frac{a}{2}}^{\pi}\left[f_{1}(t)g_{2}\!\left(t\!-\!x\!+\!\frac{a}{2}\right)\!-\!f_{2}(t)g_{1}\!\left(t\!-\!x\!
  +\!\frac{a}{2}\right)\right]dt,\!\!\\
  \eta_{2}(x)=\int_{x+\frac{a}{2}}^{\pi}\left[f_{1}(t)g_{1}\!\left(t\!-\!x\!+\!\frac{a}{2}\right)\!+\!f_{2}(t)g_{2}\!\left(t\!-\!x\!
  +\!\frac{a}{2}\right)\right]dt,
\end{cases}\!\!\!\!\!
\end{align}
for $x\in I_{3}$.
\end{lemma}

Now we are in position to give the proof of Theorem \ref{th2}.

\textbf{Proof of Theorem \ref{th2}.} According to Lemmas \ref{lem5} and \ref{lem7}, the conditions $\lambda_{n,1}=\widetilde{\lambda}_{n,1}$, $\lambda_{n,2}=\widetilde{\lambda}_{n,2}$, $n\in\mathbb{Z}$, imply that
\begin{align}\label{3.18.0}
  w_{1}(x)=\widetilde{w}_{1}(x),\quad w_{2}(x)=\widetilde{w}_{2}(x),\quad a.e. \;on \;[a,\pi],
\end{align}
and
\begin{align}\label{3.18}
  q(x)=\widetilde{q}(x),\quad p(x)=\widetilde{p}(x),\quad a.e. \;on \;I_{1}\cup I_{2}.\quad\;\;
\end{align}
Since $q(x)\!=\!\widetilde{q}(x)$, $p(x)\!=\!\widetilde{p}(x)$ a.e. on $(\frac{3a}{2},\frac{\pi}{2}\!+\!\frac{a}{4})$, it follows from (\ref{3.15})
and (\ref{3.18}) that
\begin{align}\label{3.19}
  g_{1}(x)=\widetilde{g}_{1}(x),\quad g_{2}(x)=\widetilde{g}_{2}(x),\quad a.e. \;on \;[a,\frac{\pi}{2}\!+\!\frac{a}{4}).
\end{align}
Additionally, the relations (\ref{3.10})and (\ref{3.18}) yield
\begin{align}\label{3.20}
  f_{1}(x)=\widetilde{f}_{1}(x),\quad f_{2}(x)=\widetilde{f}_{2}(x),\quad a.e. \;on \;[\frac{\pi}{2}\!+\!\frac{3a}{4},\pi].\!\!
\end{align}
In view of (\ref{3.17}) and (\ref{3.19})-(\ref{3.20}), we have
\begin{align}\label{3.21}
  \eta_{1}(x)=\widetilde{\eta}_{1}(x),\quad \eta_{2}(x)=\widetilde{\eta}_{2}(x),\quad a.e. \;on \;I_{3}.\qquad\quad
\end{align}
According Lemmas \ref{lem8} along with the relations (\ref{3.18.0}) and (\ref{3.21}), we obtain
\begin{align}\label{3.22}
  q(x)=\widetilde{q}(x),\quad p(x)=\widetilde{p}(x),\quad a.e. \;on \;I_{3}.
\end{align}
Note that $[a,\pi]=I_{1}\cup I_{2}\cup I_{3}\cup(\frac{3a}{2},\frac{\pi}{2}\!+\!\frac{a}{4})$. The relations (\ref{3.18}) and (\ref{3.22}) arrive at the assertion of Theorem \ref{th2}. \qquad\qquad\qquad\qquad\qquad\qquad\qquad\, $\Box$

Based on the above discussion process, we have the following algorithm for solving Inverse Problem 1.

\textbf{Algorithm 1.}  Let the two spectra $\{\lambda_{n,j}\}_{n\in\mathbb{Z}}$, $j=1,2$ and  partial potentials $q|_{(\frac{3a}{2},\frac{\pi}{2}+\frac{a}{4})}$, $p|_{(\frac{3a}{2},\frac{\pi}{2}+\frac{a}{4})}$ be given.

(i) Construct the functions $\Delta_{1}(\lambda)$ and $\Delta_{2}(\lambda)$ by (\ref{2.15}) and (\ref{2.16});

(ii) In accordance with (\ref{2.11}) and (\ref{2.12}), find the functions $u_{1}(x)$ and $u_{2}(x)$  by the formulae
\begin{align}
 u_{1}(x)=\frac{1}{2\pi}\sum_{n=-\infty}^{\infty}\Delta_{1}(n)\exp(-inx), \qquad\qquad\nonumber\\ u_{2}(x)=\frac{1}{2\pi}\sum_{n=-\infty}^{\infty}\Big(\Delta_{2}(n)-(-1)^{n}\Big)\exp(-inx);\!\!\!\!\!\!\nonumber
\end{align}

(iii) Construct the functions $w_{1}(x)$ and $w_{2}(x)$ by the formulae (\ref{3.1}) and (\ref{3.2});

(iv) Construct the functions $q|_{[a,\frac{3a}{2}]\cup[\pi-\frac{a}{2},\pi]}$ and $p|_{[a,\frac{3a}{2}]\cup[\pi-\frac{a}{2},\pi]}$ by the formula (\ref{3.3});

(v) Construct the functions $q|_{[\pi-a,2a]\cup[\frac{\pi}{2}+\frac{3a}{4},\pi-\frac{a}{2})}$ and $p|_{[\pi-a,2a]\cup[\frac{\pi}{2}+\frac{3a}{4},\pi-\frac{a}{2})}$ by the formulae (\ref{3.6}) and (\ref{3.7});

(vi) Construct the functions $q|_{(2a,\frac{\pi}{2}+\frac{3a}{4})}$ and $p|_{(2a,\frac{\pi}{2}+\frac{3a}{4})}$ by the formulae (\ref{3.10}),  (\ref{3.11})and (\ref{3.12});

(vii) Construct the functions $q|_{[\frac{\pi}{2}+\frac{a}{4},\pi-a)}$ and $p|_{[\frac{\pi}{2}+\frac{a}{4},\pi-a)}$ by the formulae (\ref{3.10}), (\ref{3.15}), (\ref{3.16}) and (\ref{3.17}).

\qquad
	
	\noindent {\bf Acknowledgments.}
	This work was supported in part by  the National Natural Science Foundation of China (11871031) and the National Natural Science Foundation of Jiang Su (BK20201303).

\end{document}